\documentclass{amsart}
\usepackage{sansmath}
\usepackage{amssymb,amsthm,amsmath}
\usepackage[numbers,sort&compress]{natbib}
\usepackage{color}
\usepackage{graphicx}
\usepackage{tikz}
\usepackage{caption}

\title[ ]{Analytic  solutions of nonlinear elliptic  equations on rectangular tori}

\author{Yunfeng Shi}
\address[Y. Shi]{School of Mathematical Sciences,
Peking University,
Beijing,
China} \email{yunfengshi18@gmail.com}

\keywords{Analytic periodic solutions, nonlinear elliptic equations, rectangular tori, Nash-Moser iterations, Anderson localization, KAM theory.}

\theoremstyle{plain}
\newtheorem{thm}{Theorem}[section]
 
 \newtheorem{lem}[thm]{Lemma}
 
 \theoremstyle{definition}
 \newtheorem{defn}[thm]{Definition}
 \theoremstyle{remark}
 \newtheorem{rem}[thm]{Remark}
 
 \numberwithin{equation}{section}

\begin{document}


\begin{abstract}
In this paper, we consider the  nonlinear elliptic equations  on rectangular tori.
Using methods in the study of KAM theory and Anderson localization, we prove that these equations admit many \textit{analytic}   solutions.
\end{abstract}

\maketitle
\section{Introduction and main results}
  In this paper, we investigate the following  equation  on rectangular  tori
\begin{equation}\label{op}
 -\Delta u-mu+\epsilon (f(x)u^p+g(x))=0, \  x\in\prod _{i=1}^d\left(\mathbb{R}/2\pi\beta_i\mathbb{Z}\right), d\geq 1,
\end{equation}
where $\beta=(\beta_1,\cdots,\beta_d)\in \left[\frac{1}{2},1\right]^d$,   $ p>1,\epsilon\geq 0$ and $m>0$. We assume further $f(x),g(x)$ are real trigonometric polynomials satisfying $g\not\equiv0$. Let  $\mathbb{T}^{d}:=\mathbb{R}^{d}/(2\pi\mathbb{Z})^{d}$. 
Performing a change of variables, it is convenient to consider instead the following equation
\begin{equation}\label{op1}
-\Delta_\nu u-mu+\epsilon (f(x)u^p+g(x))=0, \  x\in\mathbb{T}^{d},
\end{equation}
where
$$\Delta_{\nu}:=-\sum_{i=1}^d\nu_i^{2}\frac{\partial^2}{\partial x_i^2},\  \nu=(\beta_1^{-1},\cdots,\beta_d^{-1})\in [1,2]^d.$$

Our aim of the present paper is to find periodic solutions of  (\ref{op1}).
In fact, it is easy to see $u=0$ is not a solution of (\ref{op1}) if $\epsilon\neq0$.  Hence it is meaningful to look for periodic solutions of  (\ref{op1}) with positive $\epsilon$.

We have the following main result.
\begin{thm}\label{mthm} For any $\delta>0$, there is some $\epsilon_0>0$ depending only on $\delta,f,g,m,p,d$ such that,
for $0\leq\epsilon\leq \epsilon_0$,
there exists a set   $\Omega=\Omega{(\epsilon)}\subset [1,2]^{d}$ of Lebesgue measure $\mathrm{mes} ([1,2]^{d}\setminus \Omega)\leq \delta $ such that,
for $\nu\in \Omega$,  $\mathrm{(\ref{op1})}$ admits an  analytic  solution. More precisely, for any $\nu\in \Omega$, there exists some  $u_{\nu}\in C^{\omega}(\mathbb{T}^{d},\mathbb{R})$  so that $u_{\nu}(x)$ is a solution of  $\mathrm{(\ref{op1})}$.
\end{thm}

To prove the existence of solutions for a nonlinear elliptic equation, methods of the calculus of variations, bifurcation theory and topological degree have been
used.  By contrast, we will make use of the methods developed in  KAM (Kolmogorov-Arnold-Moser)  and Anderson localization theory. To the best of our knowledge, there is no result about the existence of \textit{analytic}  solutions for nonlinear elliptic
equations  on rectangular tori.

\begin{rem}
Regarding the rectangular tori,  there have been plenty of results on the study of Schr\"odinger equations on rectangular tori.
Bourgain \cite{Bou1993} firstly  addressed the question of Strichartz estimates for Schr\"odinger equations on tori.
In \cite{BB2007}, Bourgain carried out the study of Strichartz estimates for Schr\"odinger equations on \textit{any} $3$-dimensional rectangular tori. Later, Guo-Oh-Wang \cite{GOW}
proved some new Strichartz estimates for linear Schr\"odinger equations on \textit{any} $d$-dimensional irrational tori.
  Remarkably,   Bourgain-Demeter \cite{BD2015} proved the sharp estimates for linear Schr\"odinger equations on \textit{any} irrational tori.
 Recently, Deng-Germain-Guth \cite {DGG2017} obtained the  Strichartz estimates over large time scales for the Schr\"odinger equations on \textit{generic} rectangular tori.
 This indeed inspires us to study nonlinear elliptic equations  on \textit{most} rectangular tori. So far, there are also many results about Sobolev norms growth for equations
  on rectangular tori,  see  e.g. \cite{BMAR,DCPAM,DGIMRN,CWCPAA}.
\end{rem}

 \begin{rem}
 As we will see later, the \textit{small-divisors} difficulty appears in the present work.  
 In fact, the KAM  techniques and CWB (Craig-Wayne-Bourgain) methods are powerful tools to overcome the \textit{small-divisors} difficulty.
  The KAM results, such as the existence results of quasi-periodic (or almost-periodic) solutions for nonlinear Hamiltonian PDEs have been widely studied in the literature.
  In particular, the study for $1$-dimensional PDEs has attracted a great deal of attention over years and is well understood \cite{CLSY,LYCPAM,YJDE,YCMP,LYCMP,CWCPAM,KP1996,WaCMP,BIMRN,BJFA,Kuk1987}.
  In high dimensional case, the first result was due to Bourgain \cite{Bou1998}.
  Significantly, using the so called CWB mehods,  he proved the existence of quasi-periodic solutions for a class of nonlinear Schr\"odinger equations (NLS) on $\mathbb{T}^2$.
  Later in \cite{BB2005}, by introducing techniques from the study of Anderson localization theory  in \cite{BGS2002},  Bourgain  established  existence results of quasi-periodic solutions for NLS and nonlinear wave equations (NLW) in \textit{arbitrary} dimension $d$.
  The standard KAM approach has been  extended by Eliasson-Kuksin \cite{EK2010} to NLS on \textit{arbitrary} dimensional torus $\mathbb{T}^d$.
  Recently, Wang \cite{Wang2016} considered a class of completely resonant NLS on $\mathbb{T}^d$ with supercritical nonlinearities, and also showed the existence of quasi-periodic solutions.
  Berti-Bolle dealt with  NLS \cite{BBJEMS} and NLW \cite{BBNL} on $\mathbb{T}^d$ with  finitely differentiable nonlinearity and obtained the existence of Sobolev regular quasi-periodic solutions.  Berti-Bolle made use of a modified Nash-Moser iterations together with the multi-scale analysis.
  In a latest work by Berti-Maspero \cite{BMAR}, they proved the existence of Sobolev regular quasi-periodic solutions for  the NLW and NLS on \textit{arbitrary} rectangular tori. We should remak that all existence solutions mentioned above are at most  Gevrey regular  in time or space variables.
  In \cite{Yuanar}, Yuan developed  a new KAM scheme so that he can deal with equations with normal frequencies having finite limit-points.
 In particular, Yuan proved the existence of quasi-periodic solutions for generalized Pochhammer-Chree equations on \textit{a.e.} rectangular tori.
 Very recently, Wang \cite{WMWF} proved the existence of analytic quasi-periodic  Floquet-Bloch solutions for NLS on $\mathbb{R}^d$. In fact, the present paper is also motivated  by work of Berti-Maspero \cite{BMAR}, Wang \cite{WMWF} and Yuan \cite{Yuanar}.
  \end{rem}

 \begin{rem}
 Regarding the methods,  we use mainly the Nash-Moser iterations in  \cite{BB2005} as well as the Green's function estimates developed by Bourgain \cite{B2007}.
 As is well-known, the key to the Nash-Moser iterations is the appropriate estimates on the inverses (or Green's functions) of the linearized operators.
 If we regard $(\nu_1,\cdots,\nu_d)$  as parameters, then the \textit{covariance} property (see (\ref{cov}) in section 3) holds. More importantly, in this case the Green's functions are quite similar to that in the study of Anderson localization for quasi-periodic operators on $\mathbb{Z}^d$.
 This observation may lead to good controls of the Green's functions,  
 i.e.,   off-diagonal \textit{exponential} decays of the Green's functions.
  Let us recall briefly some Anderson localization results. For $d=1$, Bourgain-Goldstein \cite{BG2000}  originally established the non-perturbative Anderson localization for general quasi-periodic Schr\"odinger operators with real analytic potentials. They introduced the powerful semi-algebraic sets methods to eliminate the resonances.  Along this line and combining with the multi-scale analysis, Bourgain \cite{Bou2002} even proved the Anderson localization for a class of ergodic operators with   skew shifts and then got the localization and  almost periodicity of the waves for some quantum kicked rotor model. When considering on $\mathbb{Z}^d$ ($d\geq 2$),  the large deviation theorem (LDT) for Green's functions can not be derived directly from the Diophantine properties of the frequencies and the semi-algebraic sets considerations.  Actually, by posing an arithmetic condition on the frequency together the matrix-valued Cartan estimate when proving the LDT, Bourgain-Goldstein-Schlag \cite{BGS2002} showed the Anderson localization for quasi-periodic Schr\"odinger operators on $\mathbb{Z}^2$. Techniques of \cite{BGS2002} were used by Bourgain and Wang to study KAM results for high dimensional PDEs as well as some spectral problems \cite{BB2005,BW2008,BW2004,Wang2008,Wang2016,WangKG}.  For $d\geq 3$,  it is difficult to impose a similar arithmetic condition on the frequencies. To overcome this problem, Bourgain \cite{B2007} introduced new methods and successfully extended results of \cite{BGS2002} to \textit{arbitrary} dimension $d$. The basic techniques of \cite{B2007} are also  semi-algebraic sets and matrix-valued Cartan estimate, but involve more delicate analysis. Recently, methods of Bourgain \cite{B2007} were used by Goldstein-Schlag-Voda \cite{GSV} to study multi-frequency quasi-periodic Schr\"odinger operators on $\mathbb{Z}$, and by  Jitomirskaya-Liu-Shi \cite{JLS} to study a class of long-range quasi-periodic operators on $\mathbb{Z}^d$  with more general ergodic transformations.  The results of \cite {JLS} can be adapted to our problem here.
 \end{rem}

This paper is organized as follows. The basic notations are introduced in section 2. A reformulation of the problem on $\mathbb{Z}^{d}$ is presented in section 3. The key ingredient, i.e., LDT for Green's functions is proved in section 4. In section 5, the main result is established by using Nash-Moser iterations.

We want to point out that the present paper is not self-contained but relies
heavily on \cite{BB2005,B2007,JLS}.
\section{Notations} We define $a\ll b$ if there is some small $\varepsilon>0$ so that $a\leq \varepsilon b$. We write $a\sim b$ if $a\ll b$ and $b\ll a$. We write $a\pm$ to denote $a\pm\varepsilon$ for some small $\varepsilon$.

For any $x\in\mathbb{R}^{d_1}$ and $X\subset\mathbb{R}^{d_1+d_2}$,  define the $x$-section of $X$ to be the set
$$X(x)=\{y\in\mathbb{R}^{d_2}:\  (x,y)\in X\}.$$


We denote by  $\lfloor x\rfloor$  the integer part of some $x\in\mathbb{R}$.

For any $x\in\mathbb{R}^d$, let $|x|=\max\limits_{1\leq i\leq d}|x_i|$. For $U_1, U\subset\mathbb{R}^d$,  we introduce
$$\mathrm{diam}(U)=\sup_{n,n'\in U}|n-n'|, \ \mathrm{dist}(m,U)=\inf_{n\in U}|m-n|,$$
and $\mathrm{dist}(U_1,U)=\inf\limits_{n\in U_1}\mathrm{dist}(n,U)$.


\section{A reformulation of the equations  on  $\mathbb{Z}^{d}$}

Notice that (\ref{op1}) can be transformed into nonlinear equations on lattice $\mathbb{Z}^{d}$ via the standard Fourier arguments. More precisely, one just needs consider the following equations
\begin{equation}\label{op3}
 F(\widehat{u})(n)=0, \  n\in\mathbb{Z}^{d},
\end{equation}
where
\begin{equation*}F(\widehat{u})(n)=\left(\sum_{i=1}^d\nu_i^2n_i^2-m\right)\widehat{u}(n)
+\epsilon (\widehat{fu^p}+\widehat{g})(n).\end{equation*}
We want to solve (\ref{op3}) and thus employ  the Nash-Moser iterations. The linearized operator of $F$ at $u$ (we write $u=\widehat{u}$ for simplicity) reads
\begin{equation*}
F_u:= D+\epsilon S_u, u\in \mathbb{C}^{\mathbb{Z}^{d}},
\end{equation*}
where
\begin{eqnarray*}
D=\mathrm{diag}\left(\sum_{i=1}^d\nu^2_in_i^2-m\right)_{n\in\mathbb{Z}^{d}}
\end{eqnarray*}
is a diagonal operator,  and
\begin{eqnarray*}
S_u(n,n')=p\widehat{fu^{p-1}}(n-n')
\end{eqnarray*}
is a T\"oplitz operator. Define for $\theta\in\mathbb{R}^d$ the following operators
\begin{equation}\label{fu}F_u(\theta)=\mathrm{diag}\left(\sum_{i=1}^d({\nu_i}n_i+\theta_i)^2-m\right)_{n\in\mathbb{Z}^d}
+\epsilon S_u.\end{equation}
Then  the  \textit{covariance} property  holds:
\begin{equation}\label{cov}F_u(\theta)(n+\overline{n},n'+\overline{n})=F_u(\theta+\overline{n}{\nu})(n,n'),\end{equation}
where $$\overline{n}=(\overline{n_1},\cdots,\overline{n_d}),\  \overline{n}{\nu}=(\overline{n}_1{\nu_1},\cdots,\overline{n}_d{\nu_d}).$$

We should remark that the \textit{small-divisors} here are 
$$\sum_{i=1}^d{\nu_i}^2n_i^2-m.$$

\section{LDT for Green's functions}
The key of the Nash-Moser iterations is to get \textit{good} estimates of $$G_\Lambda^u(\theta)=(R_\Lambda F_u(\theta)R_\Lambda)^{-1},$$  where  $R_\Lambda$ is the restriction operator on $\Lambda\subset\mathbb{Z}^{d}$, and $F_u(\theta)$ is defined by (\ref{fu}).

We call $G_\Lambda^u(\theta)$
a Green's function.

For some technical reasons, we need introduce elementary regions on $\mathbb{Z}^{d}$.
Given $N>0$, $\emptyset\neq I\subset\{1,\cdots,d\}$ and $\varsigma=(\varsigma_i)_{i\in I}\in \{<,>\}^{I}$, define
$$Q_N(I,\varsigma):=[-N,N]^{d}\setminus\{n\in\mathbb{Z}^{d}: \ n_i\varsigma_i 0, i\in I\}.$$
Let $$\mathcal{E}_N^{0}:=\{[-N,N]^{d}\}\bigcup\limits_{\emptyset\neq I\subset\{1,\cdots,d\}, \varsigma\in \{<,>\}^{I}}\{Q_N(I,\varsigma)\}$$ and
$$\mathcal{E}_N:=\bigcup\limits_{n\in\mathbb{Z}^{d},Q\in \mathcal{E}_N^{0}}\{n+Q\}.$$

We write $G_N^u(\theta)=G_\Lambda^u(\theta)$ if  $\Lambda\in \mathcal{E}_N^0$.  From (\ref{cov}), the structure of $G_N^u(\theta)$ is similar to that of Green's functions  in  Anderson localization theory for lattice quasi-periodic Schr\"odinger operators in \cite{B2007}.

Recently, Jitomirskaya-Liu-Shi \cite{JLS} extended Bourgain's results \cite{B2007} to more general ergodic transformations as well as long-range interactions case.
Actually, the proofs of \cite{JLS} can be used in the present work with some   modifications,  and then imply the LDT for $G_N^u(\theta)$. The main techniques employed here are semi-algebraic sets analysis and the matrix-valued Cartan estimate.


Firstly, we introduce some useful facts about the \textit{semi-algebraic} sets.
\begin{defn}[Chapter 9, \cite{BB2005}]
A set $\mathcal{S}\subset \mathbb{R}^n$ is called a \textit{semi-algebraic} set if it is a finite union of sets defined by a finite number of polynomial equalities and inequalities. More precisely, let $\{P_1,\cdots,P_s\}\subset\mathbb{R}[x_1,\cdots,x_n]$ be a family of real polynomials whose degrees are bounded by $d$. A (closed) semi-algebraic set $\mathcal{S}$ is given by an expression
\begin{equation}\label{smd}
\mathcal{S}=\bigcup\limits_{j}\bigcap\limits_{\ell\in{\mathfrak{L}}_j}\left\{x\in\mathbb{R}^n: \ P_{\ell}(x)\varsigma_{j\ell}0\right\},
\end{equation}
where ${\mathfrak{L}}_j\subset\{1,\cdots,s\}$ and $\varsigma_{j\ell}\in\{\geq,\leq,=\}$. Then we say that $\mathcal{S}$ has degree at most $sd$. In fact, the \textit{degree} of $\mathcal{S}$ which is denoted by $\deg(\mathcal{S})$, means the  smallest $sd$ over all representations as in (\ref{smd}).
\end{defn}



\begin{lem}[\cite{B2007}]\label{projl}
Let $\mathcal{S}\subset[0,1]^{d=d_1+d_2}$ be a semi-algebraic set of degree $\deg(\mathcal{S})=B>0$ and $\mathrm{mes}(\mathcal{S})\leq\eta>0$, where
\begin{equation*}
\log B\ll \log\frac{1}{\eta}.
\end{equation*}
Denote by $(x_1,x_2)\in[0,1]^{d_1}\times[0,1]^{d_2}$ the product variable. Fixing
$$ \eta^{\frac{1}{d}}\leq\varepsilon\ll1,$$
then there is a decomposition of $\mathcal{S}$ as
$$\mathcal{S}=\mathcal{S}_1\cup\mathcal{S}_2$$
with the following properties. The projection of $\mathcal{S}_1$ on $[0,1]^{d_1}$ has small measure
$$\mathrm{mes}_{d_1}(\mathrm{Proj}_{x_1}\mathcal{S}_1)\leq B^{C(d)}\varepsilon,$$
and $\mathcal{S}_2$ has the transversality property
$$\mathrm{mes}_{d_2}(\mathcal{L}\cap \mathcal{S}_2)\leq B^{C(d)}\varepsilon^{-1}\eta^{\frac{1}{d}},$$
where $\mathcal{L}$ is a $d_2$-dimensional hyperplane in $[0,1]^d$ s.t.,
$$\max\limits_{1\leq j\leq d_1}|\mathrm{Proj}_\mathcal{L}(e_j)|<{\varepsilon},$$
where we denote by $e_1,\cdots,e_{d_1}$ the $x_1$-coordinate vectors.
\end{lem}


\begin{lem}[Lemma 1.18, \cite{B2007}]\label{svl}
Let $\mathcal{S}\subset [0,1]^{d_1+d_2}$ be a semi-algebraic set of degree $B$ and such that
$$\mathrm{mes}_{d_1}(\mathcal{S}(x))<\eta\ \mathrm{for}\ \forall\ x\in [0,1]^{d_2}.$$
Then the set
\begin{equation*}
\left\{(y_1,\cdots,y_{2^{d_2}})\in [0,1]^{d_12^{d_2}}:\  \bigcap\limits_{1\leq i\leq 2^{d_2}}\mathcal{S}(y_i)\neq\emptyset\right\}
\end{equation*}
is semi-algebraic of degree at most $B^{C}$ and measure at most
\begin{equation*}
B^{C}\eta^{d_1^{-{d_2}}2^{-d_2(d_2-1)/2}},
\end{equation*}
where $C=C(d_1,d_2)>0$.
\end{lem}
\begin{center}
\begin{tikzpicture}[scale=1.3]
\draw[<->](0,4)--(0,0)--(4,0);
\draw (0,4) node [left] {$y$};
\draw (4,0) node [below right] {$x$};
\draw (0.2,0.1) to [out=80,in=178] (3,3);
\draw (0.2,0.3) to [out=80,in=180] (3,3.1);
\draw[dashed] (1,0)--(1,4);
\draw [](0.2,0.1)--(0.2,0.3);
\draw (1.1,2.5) node [left] {$x=x_0$};
\draw (3,3)node [below] {$\mathcal{S}$};
\draw (3,3) to [out=10,in=100] (3,3.1);
\draw [dashed,red] (0,3.1)--(3,3.1);
\draw [dashed,red](0.2,0.1)--(0,0.1);
\draw [line width=0.05cm](1,2.0)--(1,2.15);
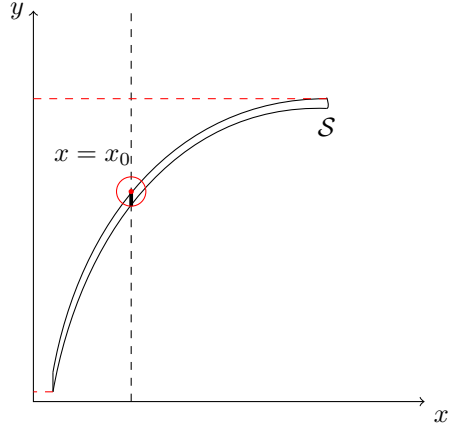
\draw (2,-0.5)node [below] {\textit{Eliminating multiple variables}};
\draw [red](1,2.15) circle [radius=0.15cm];
\draw [fill,red](1,2.15) circle [radius=0.6pt];
\end{tikzpicture}
\end{center}
\begin{lem}[Lemma 1.20, \cite{B2007}]\label{rmf}
Let $\emptyset\neq I\subset\{1,\cdots,d\}$ and
 let $\mathcal{S}\subset [0,L]^{\ell|I|}$ be a semi-algebraic set of degree $B$  such that
$$\mathrm{mes}_{\ell|I|}(\mathcal{S})<\eta.$$

For any
$$\alpha= (\alpha_i)_{i\in I}\in\mathbb{R}^{I},\ n=(n_i)_{i\in I}\in\mathbb{Z}^{I},$$
define
$$n \alpha=\left(n_i\alpha_i\right)_{i\in I}\in\mathbb{R}^{I}.$$

For any $C>1$, let $\mathcal{N}_1,\cdots,\mathcal{N}_{\ell-1}\subset \mathbb{Z}^I$ be finite sets satisfying
$$\min\limits_{ s\in I}|n_s|>(B\max\limits_{s\in I}|n'_s|)^C, $$
where $n\in \mathcal{N}_{i}, n'\in\mathcal{N}_{i-1}\  (2\leq i\leq \ell-1)$.

Let $L\sim \max\limits_{n\in\mathcal{N}_{\ell-1}}|n|$. Then there is some $C_0=C_0(\ell,|I|)>0$ such that for any $C\geq C_0$ and
$\max\limits_{n\in\mathcal{N}_{\ell-1}}|n|^C<\frac{1}{\eta}, $ one has
\begin{equation*}
\mathrm{mes}(\{\alpha\in [0,1]^{I}:\  \exists \ n^{(i)}\ \in\mathcal{N}_i\  s.t.,\  (\alpha,n^{(1)}\alpha,\cdots,n^{(\ell-1)}\alpha)\in \mathcal{S}\})\leq B^{C_0}\delta,
\end{equation*}
where
$$\delta^{-1}=\min\limits_{n\in\mathcal{N}_1}\min\limits_{s\in I}|n_s|.$$
\end{lem}

The following result was indeed proved by Bourgain \cite{B2007} in case $d=3$ and it can be easily extended to any $d\geq 1$ (see \cite{JLS} for details).
\begin{lem}[\cite{B2007}]\label{ldtp2}
Let $0<\tau_2<\tau_1<1$, $C_\star\gg1$ and $N_2\sim (\log \overline{N})^C$ for some $C>1$.
Suppose that for $\emptyset\neq I\subset\{1,\cdots,d\} $, $\theta\in \mathbb{R}^{d}$  and $\overline{N}^{\tau_2}\leq L\leq\overline{N}^{\tau_1}$,
there is no sequence $n^{(1)},\cdots,n^{(2^{d})}\in \mathbb{Z}^{{I}}\cap[-\overline{N},\overline{N}]^I$ satisfying
 \begin{eqnarray*}
\label{stp1}&&\min_{s\in I}|n_s^{(1)}|>L^{C_\star},\\
\label{stp2}&&\min_{s\in I}|n^{(i+1)}_s|\geq (L|n^{(i)}|)^{C_\star} \ (1\leq i\leq 2^{d}-1),
\end{eqnarray*}
 such that the following holds:  for all $Q_L\in \mathcal{E}_L^0$ and $1\leq i\leq 2^{d}$,  $G_{\Lambda}^{u}(\theta)$ fails $ \mathrm{(\ref{f1}) + (\ref{f2})}$ for $\Lambda=Q_L+\sum\limits_{j\in I} n^{(i)}_je_j\in\mathcal{E}_L$, where
\begin{eqnarray}
\label{f1}&&||G_{\Lambda}^u(\theta)||\leq e^{2N_2^{\kappa}}\ (0<\kappa<1),\\
\label{f2}&& |G_{\Lambda}^u(\theta)(n,n')|\leq e^{-\rho|n-n'|}\ \mathrm{for}\ |n-n'|>N_2^{1+},
\end{eqnarray}
and $|u(n)|\leq Ce^{-\rho|n|}, \rho>0$.

Suppose moreover that
$$\tau_2(C_\star+1)^{2^{(d-1)(d+2)}}\leq\tau_1.$$

Then for  $\theta\in\mathbb{R}^{d}$, the following statements hold:
\begin{itemize}
\item[(i)] There is $N$ (depending on $\theta$) satisfying $\overline{N}^{\tau_2}< N<\overline{N}^{\tau_1}$ such that if
 $$\Lambda=[-N,N]^{d}\setminus [-N^{\frac{1}{C_\star}},N^{\frac{1}{C_\star}}]^{d},$$ then
 \begin{eqnarray*}
&&||G_\Lambda^u(\theta)||\leq e^{3N_2^\kappa},\\
&&|G_\Lambda^u(\theta)(n,n')|\leq e^{-(\rho-\overline{N}^{-\frac{\tau_2}{2}})|n-n'|}\ \mathrm{for}\ |n-n'|>N_2^{1+}.
 \end{eqnarray*}
\item[(ii)]Let $Q_{\overline{N}}\in\mathcal{E}_{\overline{N}}^0$. Then for any $m\in Q_{\overline{N}}$, there exist some  $\overline{N}^{\tau_2}< N<\overline{N}^{\tau_1}$ and $\Lambda_1,\Lambda\subset Q_{\overline{N}}$ so that  $$\mathrm{diam}(\Lambda)\sim N, \mathrm{diam}(\Lambda_1)\sim N^{\frac{1}{2d}}, \ m\in \Lambda_1\subset \Lambda,\  \mathrm{dist}(m,Q_{\overline{N}}\setminus \Lambda)\geq\frac{N}{2},$$
    and
     \begin{eqnarray*}
&&||G_{\Lambda\setminus \Lambda_1}^u(\theta)||\leq e^{3N_2^\kappa},\\
&&|G_{\Lambda\setminus \Lambda_1}^u(\theta)(n,n')|\leq e^{-(\rho-\overline{N}^{-\frac{\tau_2}{2}})|n-n'|}\ \mathrm{for}\ |n-n'|>N_2^{1+}.
 \end{eqnarray*}
\end{itemize}
\end{lem}
Finally, we introduce the powerful matrix-valued Cartan estimate with $1$-dimensional parameters. For a generation to several variables  case, we refer to \cite{JLS}.

\begin{lem}[Matrix-valued Cartan estimate,\ \cite{BGS2002,B2007}]\label{mcl}
Let $T(x)$ be a self-adjoint $N\times N$ matrix function of a parameter $x\in[-\delta,\delta]$  satisfying the following conditions:
\begin{itemize}
\item[(i)] $T(x)$ is real analytic in $x\in [-\delta,\delta]$ and has a holomorphic extension to
\begin{equation*}
\mathcal{D}_{\delta,\delta_1}=\{x\in\mathbb{C}: \ |\Re x|\leq\delta,\ |\Im x|\leq \delta_1\}
\end{equation*}
satisfying
\begin{equation*}
\sup_{x\in \mathcal{D}_{\delta,\delta_1}}||T(x)||\leq B_1.
\end{equation*}
\item[(ii)]  For all $x\in[-\delta,\delta]$, there is subset $V\subset [1,N]$ satisfying the length
\begin{equation*}|V|<M,\end{equation*}
and
\begin{equation*}
||(R_{[1,N]\setminus V}T(x)R_{[1,N]\setminus V})^{-1}||\leq B_2.
\end{equation*}
\item[(iii)]
\begin{equation*}
\mathrm{mes}\{x\in[-{\delta}, {\delta}]: \ ||T^{-1}(x)||\geq B_3\}\leq 10^{-3}\delta_1(1+B_1)^{-1}(1+B_2)^{-1}.
\end{equation*}
Let
\begin{equation*}
0<\varepsilon\leq (1+B_1+B_2)^{-10 M}.
\end{equation*}
\end{itemize}
Then
\begin{equation*}
\mathrm{mes}\left\{x\in\left[-\frac{\delta}{2}, \frac{\delta}{2}\right]:\  ||T^{-1}(x)||\geq \varepsilon^{-1}\right\}\leq C\delta
 e^{-\frac{c\log \varepsilon^{-1}}{M\log(M+B_1+B_2+B_3)}}.
\end{equation*}
\end{lem}

We are  ready to prove LDT for Green's functions.

We first  prove the LDT for $G_N^u(\theta)$ if $N$ is not too large.
\begin{lem}\label{inl}
Define for $\delta>0$ the set $I=\{x\in\mathbb{R}:\ |(x+a)^2+b|<\delta\}$, where $a,b\in\mathbb{R}$. Then  $\mathrm{mes}(I)\leq C\delta^{\frac{1}{2}}$, where $C>0$ is an absolute constant.
\end{lem}
\begin{proof}
The proof is trivial and we omit the details here.
\end{proof}

\begin{lem}\label{ildt}
Let $\gamma\in (0,1),\rho>0$. Then there exist $N_0\gg1, \Sigma>1,\kappa\in (0,1)$ so that,  for any $u$ satisfying $|u(n)|\leq Ce^{-\rho|n|}$,  $N_0\leq N\leq |\log\epsilon|^\Sigma$ and any $\nu\in[1,2]^{d}$,
\begin{eqnarray*}
&& ||G_{N}^u(\theta)||\leq e^{N^{\kappa}},\\
&& |G_N^u(\theta)(n,n')|\leq e^{-\frac{4\rho}{5}|n-n'|}\  {\mathrm{for} \ |n-n'|\geq \frac{N}{10}},
\end{eqnarray*}
where $\theta$ is outside a set $X_N\subset\mathbb{R}^{d}$ satisfying for any  $1\leq j\leq d,$
\begin{equation*}\sup_{\theta_j^\neg\in\mathbb{R}^{d-1}}\mathrm{mes}(X_N(\theta_j^\neg))\leq e^{-N^\gamma},
\end{equation*}
with
$$\theta_j^\neg=(\theta_1,\cdots,\theta_{j-1},\theta_{j+1},\cdots,\theta_{d}).$$
\end{lem}

\begin{proof}
The proof is based on Lemma \ref{inl} and a standard Neumann series argument.
\end{proof}

We define the following statements ${\bf (S)}_N$ for $N\in\mathbb{N}$.
\begin{defn}[${\bf (S)}_N$]
Let $\kappa,\gamma,\zeta\in (0,1)$. There is some semi-algebraic set $\Omega_N\subset[1,2]^{d}$ of degree at most $N^{4d}$ such that for $\nu\in\Omega_N$,
there exists some $X_N\subset\mathbb{R}^{d}$ so that  for any $1\leq j\leq d$,
\begin{equation*}\sup_{\theta_j^\neg\in\mathbb{R}^{d-1}}\mathrm{mes}(X_N(\theta_j^\neg))\leq e^{-N^{\gamma}},\end{equation*}
and for $\theta\notin X_N$,
\begin{eqnarray}
\label{ldt1}&& ||G_{N}^{u_N}(\theta)||\leq e^{N^{\kappa}},\\
\label{ldt2}&& |G_N^{u_N}(\theta)(n,n')|\leq e^{-(\rho_N-)|n-n'|}\  {\mathrm{for} \ |n-n'|\geq \frac{N}{10}},
\end{eqnarray}
where $\rho_N>0$ and $u_N\in\mathbb{C}^{\mathbb{Z}^{d}}$ is rational in $\nu$ so that
$$\deg(u_N)\leq e^{(\log N)^4},\  |u_N(n)|\leq Ce^{-\rho_N|n|}.$$
Moreover, $\Omega_N\subset\Omega_{N_1}\cap\Omega_{N_2}$ with  ${N_1}\sim (\log N)^{\frac{2}{\gamma^2}}, N_2\sim(\log N)^{\frac{4}{3\gamma}}<N$  and $$\mathrm{mes}((\Omega_{N_1}\cap\Omega_{N_2})\setminus\Omega_N)\leq N^{-\zeta}.$$
\end{defn}

The main result of this section is the following LDT for Green's functions.

\begin{thm}[LDT]\label{ldt}
There are $\zeta,\kappa\in (0,1)$ such that the following  holds:  Assume $|\log\log \epsilon|\geq N_0$ and $({\bf S})_N$ holds for $N_0\leq N<\overline{N}, \ \overline{N}\geq |\log\epsilon|^\Sigma$.  
Then $({\bf S})_{\overline{N}}$ holds with $$u_{\overline{N}}=u_{{N_2}},\ {N_2}\sim (\log \overline{N})^{\frac{4}{3\gamma}}.$$
Moreover, the estimates $\mathrm{(\ref{ldt1})}$  and $\mathrm{(\ref{ldt2})}$ (for $N=\overline{N}$) remain valid if $u_{\overline{N}}$ is replaced by some $u$ satisfying:  $u$ is rational in $\nu$ and
$$\deg(u)\leq e^{(\log \overline{N})^4}, ||u-u_{\overline{N}}||_{\ell^2}\leq e^{-\rho_{{N_2}}(\log \overline{N})^{\frac{1}{\gamma^3}}}.$$
\end{thm}

\begin{proof}[\textbf{Proof of Theorem \ref{ldt}}]
 Defining scales $$N_1=\left\lfloor(\log \overline{N})^{\frac{2}{\gamma^2}}\right\rfloor,\ N_2:=\left\lfloor N_1^{\frac{2\gamma}{3}}\right\rfloor, $$ then $$\overline{N}\sim e^{N_1^{\frac{\gamma^2}{2}}}.$$   If we assume  $$|\log\log \epsilon|\geq N_0,$$ and $({\bf S})_N$ holds for $N_0\leq N<\overline{N}, \ \overline{N}\geq (\log\lambda)^\Sigma$, then
  $N_1>N_2\geq N_0$, which shows $({\bf S})_N$ holds for $N=N_1,N_2$.
Similar to \cite{B2007}, the set in $(\nu,\theta)$ defined by (\ref{ldt1}) and (\ref{ldt2})  can be replaced by  a semi-algebraic set of degree at most $e^{C(\log N)^4}$.

Consider now any scale $L$ with $$N_1<L<\overline{N}.$$ Let $S_L$ be the set of all $(\nu,\theta)\in \mathbb{R}^{d}\times\mathbb{R}^{d}$ such that for any $n\in [-L,L]^{d}$, $\nu\in \Omega_{N_1}\cap\Omega_{N_2}, \theta+n \nu\in X_{N_2}$.
 Direct computations gives
\begin{equation}\label{ldt5}\deg(S_L)\leq Ce^{C(\log N_1)^4}(2L+1)^{d}N_1^{{C}}\leq L^{C}.\end{equation}
 Fix  $1\leq j\leq d$ and $\theta_j^\neg\in \mathbb{R}^{d-1}$.
 Thus we have
\begin{equation}\label{ldt6}
\mathrm{mes}(\mathbb{R}\setminus  S_L(\theta_{j}^\neg))\leq C(2L+1)^{d}e^{-N_2^{\gamma}}\leq e^{-\frac{N_2^\gamma}{2}}.\end{equation}
Moreover, if $(\nu,\theta)\in S_L$,
 by using the resolvent identity (see \cite{JLS} for details),
one has
\begin{eqnarray*}
&& ||G^{u_{N_2}}_L(\theta)||\leq e^{2N_2^{\kappa}},\\
&&|G^{u_{N_2}}_L(\theta)(n,n')|\leq e^{-\rho_{1} |n-n'|}\ \mbox{for $|n-n'|>N_2^{1+}$},
\end{eqnarray*}
where $$\rho_{1}= \rho_{N_2}-{N_2^{0-}}.$$

In the following,  we will eliminate the variable $\theta$. This needs make further restrictions on $\nu$. By fixing $I\subset\{1,\cdots,d\}$,
 define
\begin{equation}\label{ldt9}
\mathcal{A}:=\left\{(\nu,\theta,y)\in \mathbb{[1,2]}^{d}\times \mathbb{R}^{d}\times \mathbb{R}^{I}:\  \nu\in \Omega_{N_2}\cap\Omega_{N_1}, (\nu,(\theta_j+y_j)_{j\in {I}},(\theta)_{j\notin {I}})\notin S_L\right\}.
\end{equation}
Obviously, by (\ref{ldt5}) and (\ref{ldt9}),
\begin{equation}\label{ldt10}\deg(\mathcal{A})\leq L^{C}.\end{equation}
Fix $\nu$ and consider
$$\mathcal{A}_1:=\mathcal{A}(\nu)\subset\mathbb{R}^{d}\times\mathbb{R}^{I}.$$

We note that $G_{\overline{N}}^{u_{N_2}}(\theta)$ satisfies off-diagonal exponential decay property for all $\nu\in [1,2]^d$ and $|\theta|\gtrsim \overline{N}$.
Without loss of generality, it suffices to assume
 $\mathcal{A}_1\subset [0,C\overline{N}]^d\times [0,C\overline{N}]^{I}$.
From (\ref{ldt6}), for all $\theta$,
\begin{equation}\label{ldt11}
\mathrm{mes}(\mathcal{A}_{1}(\theta))\leq \eta:=e^{-\frac{N_2^\gamma}{2}}.
\end{equation}
By (\ref{ldt10}), (\ref{ldt11}) and  Lemma \ref{svl},
\begin{equation*}
\mathcal{A}_{2}=\left\{(y_i)_{1\leq i\leq 2^{d}}:\  \bigcap_{1\leq i\leq 2^{d}}\mathcal{A}_{1}(y^{(i)})\neq \emptyset\right\}\subset[0,C\overline{N}]^{|I|2^{d}}.
\end{equation*}
 is a semi-algebraic set of degree
\begin{equation}\label{ldt12}\deg(\mathcal{A}_{2})\leq L^{C}\end{equation}
and measure
\begin{equation}\label{ldt13}
\mathrm{mes}(\mathcal{A}_{2})\leq\eta_1:= \overline{N}^C\eta^{|I|^{-d}2^{-d(d-1)/2}}.
\end{equation}
Notice that for $N_0\gg1$ and $C_\star\gg1$,
\begin{equation*}
\frac{1}{\eta_1}\gg\overline{N}^{C_\star}.
\end{equation*}
 Then from  (\ref{ldt12}), (\ref{ldt13})
and  Lemma \ref{rmf},  the set $\mathcal{A}_{3}\subset \mathbb{[1,2]}^{I}$ containing $\nu_{I}:=(\nu_j)_{j\in I}$, which is defined by the following:
\textit{there is  some sequence} $n^{(1)},\cdots,n^{(2^{d})}\in \mathbb{Z}^{{I}}\cap[-\overline{N},\overline{N}]^I$ \textsl{satisfing}
\begin{eqnarray}
\label{stp1}&&\min_{s\in I}|n_s^{(1)}|>L^{C_\star},\\
\label{stp2}&&\min_{s\in I}|n^{(i+1)}_s|\geq (L|n^{(i)}|)^{C_\star} \ (1\leq i\leq 2^{d}-1),
\end{eqnarray}
\textsl{such that}
\begin{equation*}
(\nu_{I},n^{(1)}\nu_{I},\cdots,n^{(2^{d})}\nu_{{I}})\in\mathcal{A}_{2},
\end{equation*}
satisfies
\begin{eqnarray}\label{ldt14}
&& \mathrm{mes}(\mathcal{A}_{3})\leq L^{-C_\star}L^{C},
\end{eqnarray}
where $C_\star\gg C $.
It is easy to see the total number of  $n^{(1)},\cdots,n^{(2^{d})}\in \mathbb{Z}^{{I}}\cap[-\overline{N},\overline{N}]^I$ satisfying (\ref{stp1}) and (\ref{stp2}) can be bounded by $(2\overline{N}+1)^{2d}$.
 Recalling (\ref{ldt12}), $\mathcal{A}_{3}$ is a semi-algebraic set of degree
\begin{equation}\label{ldt15}
\deg(\mathcal{A}_{3})\leq C(2\overline{N}+1)^{2d}L^{C}\leq \overline{N}^{3d},
\end{equation}
if
\begin{equation*}
\log L\leq c_1\log \overline{N},
\end{equation*}
where $0<c_1=c_1(d)\ll1$.
Define $$\Omega_{\overline{N}}:=\bigcap\limits_{\emptyset\neq I\subset\{1,\cdots,d\}}\left\{\nu\in\Omega_{N_2}\cap\Omega_{N_1}:\  \nu_{{I}}\notin\mathcal{A}_{3}\right\}.$$
If we assume
\begin{equation*}
\log L\geq {c_2}\log\overline{N},
\end{equation*}
 then by (\ref{ldt14}) and (\ref{ldt15}), for $C_\star\gg C$,
 \begin{eqnarray*}
&& \mathrm{mes}((\Omega_{N_2}\cap\Omega_{N_1})\setminus\Omega_{\overline{N}})\leq C({d})L^{-C_\star}L^C\leq \overline{N}^{-\zeta},\\
&& \deg(\Omega_{\overline{N}})\leq C(d)L^CN_{2}^{4d}N_{1}^{4d}\overline{N}^{3d}\leq \overline{N}^{4d},
\end{eqnarray*}
where $c_2>0$ will be specified below and $\zeta$ depends on $c_2$.

Let $$c_2(C_\star+1)^{2^{(d-1)(d+2)}}=c_1.$$ Then for $\nu\in\Omega_{\overline{N}}$, the assumptions of Lemma \ref{ldtp2} are satisfied.

We then construct $X_{\overline{N}}$ and finish the proof. Again, it suffices to restrict the considerations on $B(0,3\overline{N})=\{\theta\in\mathbb{R}^d:\ |\theta|\leq 3\overline{N}\}$.  As in \cite{B2007}, on each unit cube in $B(0,3\overline{N})$, using Lemma \ref{mcl}, Lemma \ref{ldtp2} and the resolvent identity,  one can find such $X_{\overline{N}}$ for $\nu\in\Omega_{\overline{N}}$.
We refer to \cite{JLS} (see also \cite{B2007}) for details.
\end{proof}

\begin{rem}\label{krem}
From the definition of $\Omega_{\overline{N}}$, the set $\Omega_{\overline{N}}$ is basically stable under perturbations of order $e^{-(\log \overline{N})^{\frac{1}{\gamma^3}}}, \ 0<\gamma\ll1.$  More precisely, one can replace $\Omega_{\overline{N}}$ by the set $$\bigcup_{i: I_i\cap\Omega_{\overline{N}}\neq\emptyset}I_i,$$ where the union runs over a partition of $[1,2]^{d}$ of cubes $I_i$ of side length $e^{-(\log \overline{N})^{\frac{1}{\gamma^3}}}$. This leads to a reformulation of Theorem \ref{ldt}, i.e., Theorem \ref{ldtr} below. This point will be useful in Section \ref{pfl}.
\end{rem}

We fix a large integer $A$ satisfying
$$ 1\ll N_0\leq A\leq  |\log\epsilon|^{\Sigma}.$$
\begin{thm}\label{ldtr}
Let $0<\epsilon\ll1$ and $r\geq 1$. Then there exists a collection $\Gamma_r$ of cubes $I$ (in $[1,2]^{d}$) of side length $e^{-(r\log A)^{\frac{1}{\gamma^3}}}$ satisfying the following:
\begin{itemize}
\item[(i)]If $\nu\in I\in\Gamma_r$, then there exists some  set $X_{A^r}$  such that, for $1\leq j\leq d$, 
     \begin{equation*}\mathrm{mes}( X_{A^r}(\theta_{j}^\neg))\leq e^{-A^{r\gamma}},\end{equation*}
and for $\theta\notin X_{A^r}$,
\begin{eqnarray*}
&& ||G_{A^r}^{u_{A^r}}(\theta)||\leq e^{A^{\kappa r}},\\
&& |G_{A^r}^{u_{A^r}}(\theta)(n,n')|\leq e^{-\frac{\rho}{2}|n-n'|}\  {\mathrm{if} \ |n-n'|\geq \frac{A^r}{10}}.
\end{eqnarray*}

\item[(ii)] Each $I\in\Gamma_r$ is contained in a cube $I'\in\Gamma_{r-1}$ and
\begin{eqnarray*}
&& \Gamma_r=\{[1,2]^{d}\}\ (1\leq r\leq r_\star),\\
&&\mathrm{mes}\left(\bigcup_{I'\in\Gamma_{r-1}}I'\setminus\bigcup_{I\in\Gamma_r}I\right)\leq A^{-\frac{\zeta r}{2}} \ (r>r_\star),
\end{eqnarray*}
where $r_\star=\left\lfloor\frac{\log|\log\epsilon|^\Sigma}{\log A}\right\rfloor$.
\end{itemize}
\end{thm}

\begin{proof}
We refer to \cite{JLS} for a detailed proof.
\end{proof}

\section{The proof of Main theorem: Nash-Moser algorithm}\label{pfl}
From Nash-Moser algorithm, the approximate solution $q_{r+1}$ at  $(r+1)^\mathrm{th}$  step can be derived from
$$q_{r+1}=q_r+\Delta_{r+1}q,$$
where the correction $\Delta_{r+1} q$ satisfies for $N=A^{r+1}$,
$$\Delta_{r+1} q=-G_{N}^{q_r}(0)F(q_r),$$
whenever $G_{N}^{q_r}(0)$  is {\textit{good}}, i.e., $G_{N}^{q_r}(0)$ satisfies the following estimates
\begin{eqnarray*}
&& ||G_{N}^{q_r}(0)||\leq A^{(r+1)^C},\\
&& |G_N^{q_r}(0)(n,n')|\leq e^{-c|n-n'|}\  {\mathrm{for} \ |n-n'|\geq (r+1)^C},
\end{eqnarray*}
where for some $C,c>0$.

As in \cite{BB2005},  to prove Theorem \ref{mthm},  it suffices to establish the following iteration arguments.
\begin{thm}
For small $\epsilon$, there exists some $A(\epsilon)\gg1$ such that  for any  $r\geq 1$ there is some $q_r(\nu)\in\mathbb{C}^{\mathbb{Z}^{d}}$ satisfying the following:
\begin{itemize}
\item[(i)] $\mathrm{supp}(q_r)\subset B(0,A^r)$.
\item[(ii)]  $$\sup_{\nu\in [1,2]^{d}}||\Delta_r q||_{\ell^2}\leq \sigma_r,\ \sup_{\nu\in [1,2]^{d}}||\partial_{\nu}\Delta_rq||_{\ell^2}\leq \overline{\sigma}_r, $$
where $$\log\log\frac{1}{\sigma_r+\overline{\sigma}_r}\sim r.$$
\item[(iii)] $$|q_r(n)|\leq Ce^{-\rho|n|}, \rho>0.$$
\item[(iv)] There is a collection $\Gamma_r$ of intervals $I$ in $\mathbb{R}^{d}$ of side length $A^{-r^C}$ so that the following holds:
\begin{itemize}
\item[(a)] On $I\in \Gamma_r$, $q_r(\nu,n)$ is given by a  rational function in $\nu$ of degree at most $A^{r^3}$.
\item[(b)] For $\nu\in\bigcup\limits_{I\in\Gamma_r}I$,
$$||F(q_r)||_{\ell^2}\leq \mu_r,\ ||\partial_\nu F(q_r)||_{\ell^2}\leq \overline{\mu}_r, $$
where $$\log\log\frac{1}{\mu_r+\overline{\mu}_r}\sim r.$$
\item[(c)] For $\nu\in\bigcup\limits_{I\in\Gamma_r}I$ and $N=A^r$,
\begin{eqnarray*}
&& ||G_{N}^{q_r}(0)||\leq A^{r^C},\\
&& |G_N^{q_r}(0)(n,n')|\leq e^{-\frac{\rho}{2}|n-n'|}\  {\mathrm{for} \ |n-n'|\geq r^C}.
\end{eqnarray*}
\item[(d)] Each $I\in\Gamma_r$ is contained in an interval $I'\in\Gamma_{r-1}$ and
\begin{eqnarray*}
&& \mathrm{mes}\left([1,2]^{d}\setminus\bigcup_{I\in\Gamma_1}I\right)\leq A^{-\frac{\zeta}{10}},\\
&&\mathrm{mes}\left(\bigcup_{I'\in\Gamma_{r-1}}I'\setminus\bigcup_{I\in\Gamma_r}I\right)\leq A^{-\frac{\zeta r}{10}},\ r\geq 2.
\end{eqnarray*}
\end{itemize}
\item[(v)]
\begin{eqnarray*}
&& \sigma_r<\sqrt{\epsilon} A^{-(\frac{4}{3})^r},\ \mu_r<\sqrt{\epsilon} A^{-(\frac{4}{3})^{r+2}},\\
&&\overline{\sigma}_r<\sqrt{\epsilon} A^{-\frac{1}{2}(\frac{4}{3})^r},\ \overline{\mu}_r<\sqrt{\epsilon} A^{-\frac{1}{2}(\frac{4}{3})^{r+2}}.
\end{eqnarray*}
\end{itemize}
\end{thm}

\begin{proof}
We start from $q_0=0$ and note that $F{(q_0)}=\epsilon\widehat{g}$.  We will construct $q_1$ firstly. It is easy to see
$$F_{q_0}(n,n)=\mathrm{diag}\left(\sum_{i=1}^d\nu_i^2n_i^2-m\right)_{n\in\mathbb{Z}^{d}} $$
and $F_{q_0}(n,n')=0$ for $n\neq n'$.
Let $A\gg1$. Due to $m>0$, then for any $0<\delta<m$, one has
$$|F_{q_0}(n,n)|>\delta\ \mathrm{for}\  |n|\leq A$$
where ${\nu}$ is outside a set $R_1$ of measure
$$\mathrm{mes}(R_1)\leq CA^{d}{\delta}^{\frac{1}{2}}.$$
Thus if ${\nu}\notin R_1$,
$$||\Delta_1q||_{\ell^2}\leq \epsilon\delta^{-1}||\widehat{g}||_{\ell^2}.$$
Using standard Neumann series arguments, one can prove the theorem for any $1\leq r\leq K\sim |\log\epsilon|^{\frac{\Sigma}{2}}$.

We will prove this theorem is true for any $r>K$. This can be completed by using LDT and make further restrictions on $\nu$. Assume $q_{r'}$ ($r'\leq r, r>K$) have been constructed and fulfill all the properties in (i)--(v). We want to construct $q_{r+1}$ and this needs to study the inverse of $R_{N}F_{q_r}(0)R_N$, where $N=A^{r+1}$. 

Firstly, in view of $||q_r-q_{r-1}||_{\ell^2}\leq \sigma_r\ll e^{-A^r}$, a standard perturbation argument implies for $\nu\in I\in \Gamma_{r}$
\begin{eqnarray}
\label{nm1}&& ||G_{A^r}^{q_r}(0)||\leq A^{r^C},\\
\label{nm2}&& |G_{A^r}^{q_r}(0)(n,n')|\leq e^{-\frac{\rho}{2}|n-n'|}\  {\mathrm{for} \ |n-n'|\geq r^C}.
\end{eqnarray}
Then we consider in $U=\{n\in\mathbb{Z}^{d}: \ \frac{A^r}{2}\leq |n|\leq N\}$. To prove  $G_U^{q_r}(0)$ has estimates (\ref{nm1}) (\ref{nm2}) with $Q_{A^r}$ being replaced by $U$, it needs make further restrictions on $\nu$ by using LDT and some semi-algebraic sets analysis arguments. Let $M_0=A^{r_0}$ satisfy
$$M_0\sim (\log N)^{\frac{4}{3\gamma}}.$$
Fix $I\in \Gamma_{r_0}$ and consider the following set
$$\mathcal{S}_N=\{(\nu,\theta)\in I\times [-N,N]^{d}: \ \nu\in I, G_{M_0}^{q_{r_0}}(\theta)\  {\mathrm{is\ not}\  \mathrm{\textit{good}} }\}.$$
Obviously, $\mathcal{S}_N$ is a semi-algebraic set of degree at most $e^{C(\log M_0)^4}$.
By Theorem \ref{ldtr}, we have
$$\mathrm{mes}(\mathcal{S}_N)\leq C N^{d}e^{-M_0^\gamma}\leq e^{-\frac{1}{2}M_0^\gamma}.$$
Using Lemma \ref{projl} as in \cite{B2007} (see also \cite{JLS}), the set $$\mathcal{S}_\star=\{\nu\in I:\ \exists \  n\in U\ s.t.,\  (\nu,n\nu)\in \mathcal{S}_N\}$$
has measure at most $A^{-{\frac{r}{2}}}$.

Sum over $I\in\Gamma_{r_0}$ and define $\Gamma_{r+1}$ to be the collection of cubes of side length $A^{-(r+1)^C}$ satisfying the following: elements of $\Gamma_{r+1}$ are derived from dividing $I'\in\Gamma_r$ into cubes $I$ of side length $A^{-(r+1)^C}$ so that $I\cap (\mathbb{R}^{d}\setminus\mathcal{S_\star})\neq\emptyset$.

Finally, on $I\in \Gamma_{r+1}$, using $||q_r-q_{r_0}||_{\ell^2}\ll e^{-(\log M_0)^{\frac{1}{\gamma^3}}}$ and the resolvent identity  implies the \textit{good} Green's function $G_{N}^{q_r}(0)$. The remainder then becomes clear and we refer to Chapter 18 of \cite{BB2005},  or \cite{BW2008} for details.
\end{proof}


\bibliographystyle{abbrv} 

\begin{thebibliography}{10}

\bibitem{BBNL}
M.~Berti and P.~Bolle.
\newblock Sobolev quasi-periodic solutions of multidimensional wave equations
  with a multiplicative potential.
\newblock {\em Nonlinearity}, 25(9):2579--2613, 2012.

\bibitem{BBJEMS}
M.~Berti and P.~Bolle.
\newblock Quasi-periodic solutions with {S}obolev regularity of {NLS} on {$\Bbb
  T^d$} with a multiplicative potential.
\newblock {\em J. Eur. Math. Soc.}, 15(1):229--286, 2013.

\bibitem{BMAR}
M.~Berti and A.~Maspero.
\newblock Long time dynamics of {S}chr\"odinger and wave equations on
  rectangular tori.
\newblock {\em arXiv: 1811.06714}, 2018.

\bibitem{Bou1993}
J.~Bourgain.
\newblock Fourier transform restriction phenomena for certain lattice subsets
  and applications to nonlinear evolution equations. {I}. {S}chr\"{o}dinger
  equations.
\newblock {\em Geom. Funct. Anal.}, 3(2):107--156, 1993.

\bibitem{BIMRN}
J.~Bourgain.
\newblock Construction of quasi-periodic solutions for {H}amiltonian
  perturbations of linear equations and applications to nonlinear {PDE}.
\newblock {\em Internat. Math. Res. Notices}, (11):475ff., approx. 21 pp.\,
  1994.

\bibitem{Bou1998}
J.~Bourgain.
\newblock Quasi-periodic solutions of {H}amiltonian perturbations of 2{D}
  linear {S}chr\"{o}dinger equations.
\newblock {\em Ann. of Math. (2)}, 148(2):363--439, 1998.

\bibitem{Bou2002}
J.~Bourgain.
\newblock Estimates on {G}reen's functions, localization and the quantum kicked
  rotor model.
\newblock {\em Ann. of Math. (2)}, 156(1):249--294, 2002.

\bibitem{BB2005}
J.~Bourgain.
\newblock {\em Green's function estimates for lattice {S}chr\"odinger operators
  and applications}, volume 158 of {\em Annals of Mathematics Studies}.
\newblock Princeton University Press, Princeton, NJ, 2005.

\bibitem{BJFA}
J.~Bourgain.
\newblock On invariant tori of full dimension for 1{D} periodic {NLS}.
\newblock {\em J. Funct. Anal.}, 229(1):62--94, 2005.

\bibitem{B2007}
J.~Bourgain.
\newblock Anderson localization for quasi-periodic lattice {S}chr\"odinger
  operators on {$\Bbb Z^d$}, {$d$} arbitrary.
\newblock {\em Geom. Funct. Anal.}, 17(3):682--706, 2007.

\bibitem{BB2007}
J.~Bourgain.
\newblock On {S}trichartz's inequalities and the nonlinear {S}chr\"{o}dinger
  equation on irrational tori.
\newblock In {\em Mathematical aspects of nonlinear dispersive equations},
  volume 163 of {\em Ann. of Math. Stud.}, pages 1--20. Princeton Univ. Press,
  Princeton, NJ, 2007.

\bibitem{BD2015}
J.~Bourgain and C.~Demeter.
\newblock The proof of the {$l^2$} decoupling conjecture.
\newblock {\em Ann. of Math. (2)}, 182(1):351--389, 2015.

\bibitem{BG2000}
J.~Bourgain and M.~Goldstein.
\newblock On nonperturbative localization with quasi-periodic potential.
\newblock {\em Ann. of Math. (2)}, 152(3):835--879, 2000.

\bibitem{BGS2002}
J.~Bourgain, M.~Goldstein, and W.~Schlag.
\newblock Anderson localization for {S}chr\"odinger operators on {$\bold Z^2$}
  with quasi-periodic potential.
\newblock {\em Acta Math.}, 188(1):41--86, 2002.

\bibitem{BW2004}
J.~Bourgain and W.-M. Wang.
\newblock Anderson localization for time quasi-periodic random
  {S}chr\"{o}dinger and wave equations.
\newblock {\em Comm. Math. Phys.}, 248(3):429--466, 2004.

\bibitem{BW2008}
J.~Bourgain and W.-M. Wang.
\newblock Quasi-periodic solutions of nonlinear random {S}chr\"{o}dinger
  equations.
\newblock {\em J. Eur. Math. Soc.}, 10(1):1--45, 2008.

\bibitem{CWCPAA}
F.~Catoire and W.-M. Wang.
\newblock Bounds on {S}obolev norms for the defocusing nonlinear
  {S}chr\"{o}dinger equation on general flat tori.
\newblock {\em Commun. Pure Appl. Anal.}, 9(2):483--491, 2010.

\bibitem{CLSY}
H.~Cong, J.~Liu, Y.~Shi, and X.~Yuan.
\newblock The stability of full dimensional {KAM} tori for nonlinear
  {S}chr\"{o}dinger equation.
\newblock {\em J. Differential Equations}, 264(7):4504--4563, 2018.

\bibitem{CWCPAM}
W.~Craig and C.~E. Wayne.
\newblock Newton's method and periodic solutions of nonlinear wave equations.
\newblock {\em Comm. Pure Appl. Math.}, 46(11):1409--1498, 1993.

\bibitem{DCPAM}
Y.~Deng.
\newblock On growth of {S}obolev norms for energy critical {NLS} on irrational
  tori: small energy case.
\newblock {\em Comm. Pure Appl. Math.}, 2018.

\bibitem{DGIMRN}
Y.~Deng and P.~Germain.
\newblock Growth of solutions to {NLS} on irrational tori.
\newblock {\em IMRN}, 2017.

\bibitem{DGG2017}
Y.~Deng, P.~Germain, and L.~Guth.
\newblock Strichartz estimates for the {S}chr\"{o}dinger equation on irrational
  tori.
\newblock {\em J. Funct. Anal.}, 273(9):2846--2869, 2017.

\bibitem{EK2010}
L.~H. Eliasson and S.~B. Kuksin.
\newblock K{AM} for the nonlinear {S}chr\"{o}dinger equation.
\newblock {\em Ann. of Math. (2)}, 172(1):371--435, 2010.

\bibitem{GSV}
M.~Goldstein, W.~Schlag, and M.~Voda.
\newblock On localization and the spectrum of multi-frequency quasi-periodic
  operators.
\newblock {\em arXiv: 1610.00380}, 2016.

\bibitem{GOW}
Z.~Guo, T.~Oh, and Y.~Wang.
\newblock Strichartz estimates for {S}chr\"{o}dinger equations on irrational
  tori.
\newblock {\em Proc. Lond. Math. Soc. (3)}, 109(4):975--1013, 2014.

\bibitem{JLS}
S.~Jitomirskaya, W.~Liu, and Y.~Shi.
\newblock Anderson localization for long-range quasi-periodic operators on
  $\mathbb{Z}^d$.
\newblock {\em Preprint}.

\bibitem{KP1996}
S.~Kuksin and J.~P\"{o}schel.
\newblock Invariant {C}antor manifolds of quasi-periodic oscillations for a
  nonlinear {S}chr\"{o}dinger equation.
\newblock {\em Ann. of Math. (2)}, 143(1):149--179, 1996.

\bibitem{Kuk1987}
S.~B. Kuksin.
\newblock Hamiltonian perturbations of infinite-dimensional linear systems with
  imaginary spectrum.
\newblock {\em Funktsional. Anal. i Prilozhen.}, 21(3):22--37, 95, 1987.

\bibitem{LYCPAM}
J.~Liu and X.~Yuan.
\newblock Spectrum for quantum {D}uffing oscillator and small-divisor equation
  with large-variable coefficient.
\newblock {\em Comm. Pure Appl. Math.}, 63(9):1145--1172, 2010.

\bibitem{LYCMP}
J.~Liu and X.~Yuan.
\newblock A {KAM} theorem for {H}amiltonian partial differential equations with
  unbounded perturbations.
\newblock {\em Comm. Math. Phys.}, 307(3):629--673, 2011.

\bibitem{Wang2008}
W.-M. Wang.
\newblock Pure point spectrum of the {F}loquet {H}amiltonian for the quantum
  harmonic oscillator under time quasi-periodic perturbations.
\newblock {\em Comm. Math. Phys.}, 277(2):459--496, 2008.

\bibitem{Wang2016}
W.-M. Wang.
\newblock Energy supercritical nonlinear {S}chr\"{o}dinger equations:
  quasiperiodic solutions.
\newblock {\em Duke Math. J.}, 165(6):1129--1192, 2016.

\bibitem{WangKG}
W.-M. Wang.
\newblock Quasi-periodic solutions for nonlinear {K}lein-{G}ordon equations.
\newblock {\em arXiv:1609.00309}, 2016.

\bibitem{WMWF}
W.-M. Wang.
\newblock Space quasi-periodic standing waves for nonlinear {S}chr\"odinger
  equations.
\newblock {\em arXiv: 1806.02038}, 2018.

\bibitem{WaCMP}
C.~E. Wayne.
\newblock Periodic and quasi-periodic solutions of nonlinear wave equations via
  {KAM} theory.
\newblock {\em Comm. Math. Phys.}, 127(3):479--528, 1990.

\bibitem{YCMP}
X.~Yuan.
\newblock Construction of quasi-periodic breathers via {KAM} technique.
\newblock {\em Comm. Math. Phys.}, 226(1):61--100, 2002.

\bibitem{YJDE}
X.~Yuan.
\newblock Quasi-periodic solutions of completely resonant nonlinear wave
  equations.
\newblock {\em J. Differential Equations}, 230(1):213--274, 2006.

\bibitem{Yuanar}
X.~Yuan.
\newblock {KAM} theorem with normal frequencies of finite limit-points for some
  shallow water equations.
\newblock {\em arXiv: 1809.05671}, 2018.

\end{thebibliography}

\end{document}